\documentclass[11pt, reqno]{amsart}
\usepackage{lineno,hyperref}
\usepackage{cases}
\usepackage[square, comma, sort&compress, numbers]{natbib}%参考文献连续编号
\usepackage{float}
\usepackage{pifont}
\usepackage{bbding}
\usepackage{amssymb}
\usepackage{mathrsfs}
\usepackage{subfigure}
\usepackage{caption}
\usepackage{geometry}
\usepackage{color}
\usepackage{amsmath, amsthm, amscd, amsfonts, amssymb, graphicx, color}
\usepackage{lineno}
\newtheorem{theorem}{Theorem}

\newtheorem{lemma}{Lemma}

\newtheorem{cor}{Corollary}

\numberwithin{equation}{section}

\newcommand{\abs}[1]{\left\vert#1\right\vert}

\newcommand{\C}{\mbox{$\mathbb{C}$}}

\newcommand{\D}{\mbox{$\mathbb{D}$}}

\begin{document}
\setcounter{page}{1}

\title[Initial successive coefficients]
{Initial successive coefficients for certain classes of univalent functions involving the\\ exponential function}

\author[L. Shi]{Lei Shi}
\address[Lei Shi]{School of Mathematics and Statistics, Anyang Normal University, Anyang 455002, Henan, P. R. China.}
\email{shimath$@$163.com}

\author[Z.-G. Wang]{Zhi-Gang Wang}
\address[Zhi-Gang Wang]{School of Mathematics and Computing Science, Hunan First Normal University, Changsha 410205, Hunan, P. R. China.}
\email{wangmath$@$163.com}

\author[R.-L. Su]{Ren-Li Su}
\address[Ren-Li Su]{School of Mathematics and Statistics, Changsha University of Science and Technology;  Hunan Provincial Key Laboratory of Mathematical Modeling and Analysis in Engineering,
       Changsha 410114, Hunan,   P. R. China.}
\email{1579206189$@$qq.com}

\author[M. Arif]{Muhammad Arif}
\address[Muhammad Arif]{Department of Mathematics, Abdul Wali Khan University, Mardan 23200, Pakistan.}
\email{marimaths$@$awkum.edu.pk}

%\thanks{$^*$Corresponding author.}


\date{}

\begin{abstract}
 Let $\mathcal{S}$ denote the family of all functions that are analytic and univalent in the unit disk $\D:=\{z:\abs{z}<1\}$
and satisfy $f(0)=f^{\prime}(0)-1=0$. In the present paper, we consider certain subclasses of univalent functions associated with the exponential function, and obtain the sharp upper bounds on the initial coefficients and the difference of  initial successive coefficients for functions
belonging to these classes.
\end{abstract} \maketitle
\section{Introduction}

Let $\mathcal{A}$ denote the class of functions $f$ of the form
\begin{equation}\label{f1}
f(z)=z+\sum_{n=2}^{\infty}a_{n}z^{n},
\end{equation} which are  analytic in
the unit disk $\D:=\{z:\abs{z}<1\}$
and satisfy the conditions $f(0)=f^{\prime}(0)-1=0$. Let $\mathcal{S}$ be the set of all functions $f\in\mathcal{A}$
that are univalent in $\D$. Let $\mathcal{S}^{*}$ and $\mathcal{K}$ denote the subclasses of $\mathcal{S}$  consisting of starlike functions and convex functions, respectively.

Let $\mathcal{P}$ denote the class of all functions $\mathfrak{p}(z)$ analytic and having positive real part
in $\D$, with the form
\begin{equation}
\mathfrak{p}(z)=1+\sum_{n=1}^{\infty}p_{n}z^{n}.
\end{equation}

For two functions $f$ and $g$, analytic in ${\D}$, we say that the
function $f$ is subordinate to $g$ in ${\D}$, and write
$$f(z)\prec g(z)\quad (z\in {\D}),$$ if there exists a Schwarz
function $\omega$, which is analytic in ${\D}$ with
$$\omega(0)=0\ {\rm and}\ \abs{\omega(z)}<1 \ (z \in {\D})$$ such that
$$f(z)=g\big(\omega(z)\big)\quad (z\in{\D}).$$

Using the subordination relationship, Ma and Minda \cite{ma} introduced the class of Ma-Minda type of starlike functions $\mathcal{S}^{*}(\phi)$, which is defined by
\begin{equation}
\mathcal{S}^{*}(\phi):=\left\{f\in\mathcal{A}\ {\rm{and}} \ \frac{zf^{\prime}(z)}{f(z)}\prec\phi(z)\right\},
\end{equation}
where $\phi(z)$ is analytic and univalent in $\D$ and for which $\phi(\D)$ is convex with $\phi(z)\in\mathcal{P}$ for $z\in\D$.

For a constant $\lambda$ with $0<\lambda\leq \frac{\pi}{2}$, by setting
\begin{equation}
\phi(z)=e^{\lambda z}\quad(z\in\D),
\end{equation}
we have the class $\mathcal{S}_{\lambda e}^{*}$ which is defined  by the condition
\begin{equation}\label{a1}
\mathcal{S}_{\lambda e}^{*}:=\left\{z\in\mathcal{A}\quad{\rm{and}}\quad\frac{zf^{\prime}(z)}{f(z)}\prec e^{\lambda z}\right\}\quad(z\in\D).
\end{equation}

It can be seen that the condition \eqref{a1} is equivalent to
\begin{equation}
\abs{\log \frac{zf^{\prime}(z)}{f(z)}}\leq \lambda\quad(z\in\D).
\end{equation}

Also,  we denote by $\mathcal{K}_{\lambda e}$ the class of functions $f\in\mathcal{A}$ satisfying the condition
\begin{equation}\label{a2}
1+\frac{zf^{\prime\prime}(z)}{f^{\prime}(z)}\prec e^{\lambda z}\quad(z\in\D).
\end{equation}

Let $z=re^{i\theta}$, $r\in[0,1)$, $\theta\in [0,2\pi]$, we have $\Re\left(e^{\lambda z}\right)=e^{\lambda r\cos\theta}\cos(\lambda r\sin\theta)$.
It is clear that $\cos(\lambda r\sin\theta)>0$ for $\lambda\in (0,\frac{\pi}{2}]$ and thus $\Re\left(e^{\lambda z}\right)>0\, (z\in\D)$. In view of this fact, the class $\mathcal{S}_{\lambda e}^{*}$ is a subclass of starlike functions $\mathcal{S^{*}}$ and $\mathcal{K}_{\lambda e}$ is a subclass of convex functions $\mathcal{K}$.

By choosing $\lambda=1$, we obtain the families $\mathcal{S}_{e}^{*}$ and $\mathcal{K}_{e}^{*}$ which were introduced and investigated by Mediratta {\it et al.} \cite{me} and were later studied by many authors, see \cite{KU,SU,AD,SH,HU,pa,ZH} and the references cited therein. Clearly, for $0<\zeta\leq1$ and $1\leq\eta\leq\frac{\pi}{2}$, we have
\begin{equation*}
\mathcal{S}_{\zeta e}^{*}\subseteq \mathcal{S}_{e}^{*}\subseteq \mathcal{S}_{\eta e}^{*}.
\end{equation*}

In recent years, the difference of the moduli of successive coefficients of a
function $f\in\mathcal{S}$ has attracted many researchers' attention (see \cite{HAY,LEU,YE,YE2}). Because of the triangle inequality $\abs{\abs{a_{n+1}}-\abs{a_{n}}}\leq\abs{a_{n+1}-a_{n}}$, sometimes it maybe useful to study the upper bounds of $\abs{a_{n+1}-a_{n}}$ for some refined subclasses of starlike and convex functions to obtain the upper bound of $\abs{\abs{a_{n+1}}-\abs{a_{n}}}$. In \cite{ROB}, Robertson proved that $\abs{a_{n+1}-a_{n}}\leq\frac{2n+1}{3}\abs{a_{2}-1}$ for all $f\in\mathcal{K}$. Recently,  Li and Sugawa \cite{li} studied the related problem of maximizing the function $\abs{\abs{a_{n+1}}-\abs{a_{n}}}$ with the help of  $\abs{a_{n+1}-a_{n}}$ for convex functions $f$ with $f^{\prime\prime}(0)=p$ for a prescribed $p\in [0,2]$. For some subclasses of analytic univalent functions, the sharp upper bounds of $\abs{a_3-a_2}$ and $\abs{a_4-a_3}$ were obtained by Peng and Obrandovi\'{c} \cite{peng}.

Motivated essentially by the above work, in the present paper, we aim at proving some results on the upper bounds of the initial coefficients and the difference of initial successive coefficients for $f$ belonging to the classes $\mathcal{S}_{\lambda e}^{*}$ and $\mathcal{K}_{\lambda e}$.
\section{Preliminary results}

To derive our main results, we need the following lemmas.

 \begin{lemma}{\rm{(See \cite{pr})}}\label{Lem001}
Let $\omega(z)=\sum_{n=1}^{\infty}c_{k}z^{k}$ be a Schwarz function. Then, for any real number $\mu$ and $\nu$ the following sharp estimate holds
\begin{equation}
\Psi(\omega)=\abs{c_{3}+\mu c_{1}c_{2}+\nu c_{1}^3}\leq \Phi(\mu,\nu),
\end{equation}
where $\Phi(\mu,\nu)$ is given in complete form in {\rm [\cite{pr},Lemma 2]}, and here we will only use
 \begin{equation}
\Phi(\mu,\nu)\leq\left\{\begin{array}{lccc}1 &(\mu,\nu)\in D_{1},\\ \\
\frac{2}{3}\left(\abs{\mu}+1\right)\left(\frac{\abs{\mu}+1}{3\abs{\mu}+1+\nu}\right)^{1/2} &(\mu,\nu)\in D_{2}\bigcup D_{3},\\ \\
\frac{1}{3}\nu\left(\frac{\mu^2-4}{\mu^2-4\nu}\right)\left(\frac{\mu^2-4}{3(\nu-1)}\right)^{1/2} &(\mu,\nu)\in D_{4},\\ \\
\abs{\nu}  &(\mu,\nu)\in D_{5},
\end{array}\right.
 \end{equation}
 where
 \begin{equation*}\begin{split}
&D_{1}=\left\{(\mu,\nu):\abs{\mu}\leq\frac{1}{2},-1\leq\nu\leq1\right\},\\
&D_{2}=\left\{(\mu,\nu):\frac{1}{2}\leq\abs{\mu}\leq2,-\frac{2}{3}(\abs{\mu}+1)\leq\nu\leq\frac{4}{27}\left((\abs{\mu}+1)^3-(\abs{\mu}+1)\right)\right\}, \\
&D_{3}=\left\{(\mu,\nu):\abs{\mu}\geq2,-\frac{2}{3}(\abs{\mu}+1)\leq\nu\leq\frac{2\abs{\mu}(\abs{\mu}+1)}{\mu^2+2\abs{\mu}+4}\right\},\\
&D_{4}=\left\{(\mu,\nu):2\leq\abs{\mu}\leq4,\frac{2\abs{\mu}(\abs{\mu}+1)}{\mu^2+2\abs{\mu}+4}\leq\nu\leq\frac{1}{12}(\mu^2+8)\right\}\setminus\left\{(2,1)\right\},\\
&D_{5}=\left\{(\mu,\nu):2\leq\abs{\mu}\leq4,\nu\geq\frac{1}{12}(\mu^2+8)\right\}.
  \end{split}
  \end{equation*}
 \end{lemma}

\begin{lemma}\label{0Lem1}{\rm{(See \cite{LIB,LIB2})}}
Let $-2\leq p_1\leq 2 $ and $p_2,\, p_3\in\C$. Then there  exists a function $\mathfrak{p}\in\mathcal{P}$ with
\begin{equation}\label{b1}
\mathfrak{p}(z)=1+p_1z+p_2z^2+p_3z^3+\cdots
\end{equation}
if and only if
\begin{equation}\label{b2}
2p_2=p_1^2+(4-p_1^2)x
\end{equation}
and
\begin{equation}\label{b3}
4p_{3}=p_{1}^3+2(4-p_{1}^2)p_{1}x-(4-p_{1}^2)p_{1}x^2+2(4-p_{1}^2)(1-{\abs{x}}^2)y
\end{equation}
for some $x,y\in\C$ with $\abs{x}\leq1$ and $\abs{y}\leq1$.
\end{lemma}

\begin{lemma}\label{0lem2}{\rm{(See \cite{li})}}
For given real numbers $a,b,c$, let
 \begin{equation}\label{yabc}
Y(a,b,c)=\max_{z\in\overline{{\small\D}}}(\abs{a+bz+cz^2}+1-\abs{z}^2).
 \end{equation}
If $a\geq0$ and $c\geq0$, then
\begin{equation}\label{Y}
Y(a,b,c)=\left\{\begin{array}{lc}a+\abs{b}+c &\abs{b}\geq2(1-c),\\ \\
1+a+\frac{b^2}{4(1-c)} &\abs{b}\leq2(1-c).
\end{array}\right.
\end{equation}
The maximum in the definition of $Y(a,b,c)$ is attained at $z=\pm1$ in the first case
according as $b=\pm\abs{b}$.
 \end{lemma}

 \begin{lemma}{\rm{(See \cite{ro})}}\label{Lem02}
If $\mu(z)=1+\sum_{k=1}^{\infty}\mu_{k}z^{k}$ is subordinate to $\nu=1+\sum_{k=1}^{\infty}\nu_k z^k$ in $\D$, where $\nu(z)$ is univalent in $\D$ and $\nu(\D)$ is convex, then
\begin{equation}
\abs{\mu_n}\leq \abs{\nu_{1}}\quad (n\geq1).
\end{equation}
 \end{lemma}

 The proof of the following lemma is similar to that of Lemma 2.2 in \cite{wsy}.

 \begin{lemma}\label{Lem5}
Suppose  that the sequence $\{A_{m}\}_{m=2}^{\infty}$ is defined by
 \begin{equation}\label{26}
 \begin{cases}
\displaystyle A_{m}=\lambda&\left(m=2\right),\\
\displaystyle A_{m}=\frac{\lambda}{m-1}\left(1+\sum_{k=2}^{m-1}A_{k}\right)&\left(m\geq3\right).
 \end{cases}
 \end{equation}
 Then
 \begin{equation}
 A_{m}=\frac{1}{(m-1)!}\prod_{k=0}^{m-2}(\lambda+k)\quad\left(m\geq2\right).
 \end{equation}
 \end{lemma}
\begin{proof} From \eqref{26}, we have
 \begin{equation}\label{28}
 (m-1)A_{m}=\lambda\left(1+\sum_{k=2}^{m-1}A_{k}\right)
 \end{equation}
  and
  \begin{equation}\label{29}
 mA_{m+1}=\lambda\left(1+\sum_{k=2}^{m}A_{k}\right).
 \end{equation}
 Combining \eqref{28}  and \eqref{29}, we find that
  \begin{equation}
  \frac{A_{m+1}}{A_{m}}=\frac{\lambda+m-1}{m}\quad\left(m\geq2\right).
  \end{equation}
 Thus,
 \begin{equation}\begin{split}
 A_{m}&=\frac{A_{m}}{A_{m-1}}\cdot\frac{A_{m-1}}{A_{m-2}}\cdot\cdot\cdot\frac{A_{3}}{A_{2}}\cdot A_{2}\\
 &=\frac{\lambda+m-2}{m-1}\cdot\frac{\lambda+m-3}{m-2}\cdot\cdot\cdot\frac{\lambda+1}{2}\cdot\lambda\\
 &=\frac{1}{(m-1)!}\prod_{k=0}^{m-2}(\lambda+k)\quad\left(m\geq3\right).
 \end{split}
 \end{equation}
In conjunction with \eqref{26}, we complete the proof of Lemma \ref{Lem5}.
 \end{proof}

\vskip.05in

\section{Main results}
We first discuss the absolute values of the second, third and fourth coefficient of functions in the class $\mathcal{S}_{\lambda e}^{*}$.
 \begin{theorem}\label{tt1}
Suppose that $f(z)=z+\sum_{n=2}^{\infty}a_{n}z^n\in\mathcal{S}_{\lambda e}^{*}$. Then
 \begin{equation}
\abs{a_{2}}\leq\lambda,
 \end{equation}
 \begin{equation}\label{0aa3}
\abs{a_{3}}\leq\left\{\begin{array}{lc}\frac{1}{2}\lambda &0<\lambda\leq\frac{2}{3},\\ \\
\frac{3}{4}\lambda^2 &\frac{2}{3}<\lambda\leq\frac{\pi}{2},
\end{array}\right.
 \end{equation}
 and
  \begin{equation}\label{0aa4}
 \abs{a_{4}}\leq\left\{\begin{array}{lccc}\frac{1}{3}\lambda &\lambda\in\left(0,\frac{1}{5}\right],\\ \\
\frac{1}{9}\lambda(5\lambda+2)\left(\frac{30\lambda+12}{17\lambda^2+90\lambda+12}\right)^{1/2}  &\lambda\in\left(\frac{1}{5}, r_{0}\right],\\ \\
\frac{17}{252}(25\lambda^2-16)\left(\frac{25\lambda^2-16}{17\lambda^2-12}\right)^{1/2}  &\lambda\in\left(r_{0},\sqrt{\frac{32}{43}}\,\right],\\ \\
\frac{17}{36}\lambda^3 &\lambda\in\left(\sqrt{\frac{32}{43}},\frac{\pi}{2}\right],
\end{array}\right.
 \end{equation}
where $r_{0}\thickapprox 0.8602$ is the positive root of the equation
\begin{equation}\label{eq1}
425\lambda^3+340\lambda^2-328\lambda-240=0.
\end{equation}
All the bounds are sharp.
\end{theorem}
\begin{proof}
Let $f\in \mathcal{S}_{\lambda e}^{*}$. Then we can write \eqref{a1} in terms of Schwarz function as
\begin{equation*}
\frac{zf^{\prime}(z)}{f(z)}=e^{\lambda\omega(z)}.
\end{equation*}
From \eqref{f1}, we can get
\begin{equation}\label{zf1}
\frac{zf^{\prime}(z)}{f(z)}=1+a_2z+(2a_3-{a_2}^2)z^2+(3a_4-3a_2a_3+{a_2}^3)z^3+\cdots.
\end{equation}
Suppose that
\begin{equation*}
\omega(z)=\sum_{n=1}^{\infty}c_{n}z^{n}.
\end{equation*}
From the series expansion of $e^{\lambda\omega}$ along with some calculations, we have
\begin{equation}\label{e1w}
e^{\lambda\omega(z)}=1+\lambda c_1z +\left(\frac{\lambda^2}{2} {c_1}^{2}+\lambda c_{2}\right)z^2+\left(\frac{\lambda^3}{6}c_{1}^{3}+\lambda^2 c_1c_{2}+\lambda c_3\right)z^3+\cdots.
\end{equation}
Comparing \eqref{zf1} with \eqref{e1w}, we have
\begin{equation}\label{0bb2}
a_2=\lambda c_1,
\end{equation}
\begin{equation}\label{0bb3}
a_3=\frac{\lambda}{2}\left(c_{2}+\frac{3}{2}\lambda c_{1}^2\right),
\end{equation}
and
\begin{equation}\label{0bb4}
a_4=\frac{\lambda}{3}\left(c_3+\frac{5}{2}\lambda c_{1}c_{2}+\frac{17}{12}\lambda^2 c_{1}^3\right).
\end{equation}
Since $\omega$ is a Schwarz function, we have $\abs{c_{1}}\leq1$. Hence, we obtain that
\begin{equation}
\abs{a_2}\leq \lambda.
 \end{equation}
Using a result of Carleson \cite{ca} (see also \cite{go}), we have $\abs{c_{2}}\leq1-\abs{c_{1}}^2$.  By virtue of \eqref{0bb3}, we find that
\begin{equation}\label{0a3}
\abs{a_3}\leq \frac{\lambda}{2}\left[1+(\frac{3}{2}\lambda-1)\abs{c_{1}}^2\right].
 \end{equation}
Then the inequality \eqref{0aa3} follows from \eqref{0a3} with $\abs{c_{1}}\in [0,1]$.

To obtain the sharp bound of $\abs{a_4}$, we will use Lemma \ref{Lem001} in the following. Let $\mu=\frac{5}{2}\lambda$, $\nu=\frac{17}{12}\lambda^2$. Suppose that $0<\lambda\leq\frac{1}{5}$,  we see that $(\mu,\nu)\in D_{1}$. From Lemma \ref{Lem001}, we have
 \begin{equation}
\abs{c_3+\frac{5}{2}\lambda c_{1}c_{2}+\frac{17}{12}\lambda^2}\leq1.
 \end{equation}
 By means of \eqref{0bb4}, we obtain
 \begin{equation}\label{1a4}
\abs{a_4}\leq\frac{\lambda }{3}.
 \end{equation}
For $\frac{1}{5}<\lambda\leq\frac{4}{5}$, it is clear that $\frac{1}{2}<\abs{\mu}\leq2$ and $\nu\geq-\frac{2}{3}(\abs{\mu}+1)$. Since
\begin{equation*}
\nu\leq\frac{4}{27}\left((\abs{\mu}+1)^3-(\abs{\mu}+1)\right)
\end{equation*}
is equivalent to
\begin{equation*}
750\lambda^2+441\lambda+240\geq0,
\end{equation*}
it is easily seen that $(\mu,\nu)\in D_{2}$ for $\frac{1}{5}<\lambda\leq\frac{4}{5}$.
Thus, an application of Lemma \ref{Lem001} leads to
 \begin{equation}\label{2a4}
\abs{a_4}\leq\frac{\lambda}{9}(5\lambda+2)\left(\frac{30\lambda+12}{17\lambda^2+90\lambda+12}\right)^{1/2}.
 \end{equation}
Now, we suppose that $\frac{4}{5}<\lambda\leq\frac{\pi}{2}$. It is not hard to verify that $2\leq\mu<4$. Since $\nu\leq\frac{1}{12}(\mu^2+8)$ is equivalent to $\lambda^2\leq\frac{32}{43}$, we have
  \begin{equation*}
\nu>\frac{1}{12}(\mu^2+8)
 \end{equation*}
 for $\lambda>\sqrt{\frac{32}{43}}\thickapprox0.8627$. This implies that
 $(\mu,\nu)\in D_{5}$ for $\lambda\in\left(\sqrt{\frac{32}{43}},\frac{\pi}{2}\right]$. Using Lemma \ref{Lem001}, yields to
  \begin{equation}\label{3a4}
\abs{a_4}\leq\frac{17}{36}\lambda^3.
 \end{equation}
For $\frac{4}{5}<\lambda\leq\sqrt{\frac{32}{43}}$, it is clear that $\nu\leq\frac{1}{12}(\mu^2+8)$. A simple calculation shows that
\begin{equation*}
v\geq\frac{2\abs{\mu}(\abs{\mu}+1)}{\mu^2+2\abs{\mu}+4}
 \end{equation*}
 is equivalent to
 \begin{equation*}
425\lambda^3+340\lambda^2-328\lambda-240\geq0.
 \end{equation*}
 Let $g(\lambda)=425\lambda^3+340\lambda^2-328\lambda-240$. The numerical computation  shows that the only positive root of $g(\lambda)=0$ is $r_{0}\thickapprox0.8602$. For $\frac{4}{5}<\lambda\leq r_{0}$, we have $g(\lambda)\leq0$ and hence $
v\leq\frac{2\abs{\mu}(\abs{\mu}+1)}{\mu^2+2\abs{\mu}+4}$. This means that $(\mu,\nu)\in D_{3}$ and we thus have
  \begin{equation}\label{4a4}
\abs{a_4}\leq\frac{\lambda}{9}(5\lambda+2)\left(\frac{30\lambda+12}{17\lambda^2+90\lambda+12}\right)^{1/2}.
 \end{equation}
 For $\lambda\in \left(r_{0},\sqrt{\frac{32}{43}}\,\right]$, we have $g(\lambda)\geq0$ and hence $v\geq\frac{2\abs{\mu}(\abs{\mu}+1)}{\mu^2+2\abs{\mu}+4}$. Obviously, $\mu=\frac{5}{2}\lambda\neq2$. Therefore, we know that $(\mu,\nu)\in D_{4}$ and
 \begin{equation}\label{5a4}
\abs{a_4}\leq\frac{17}{252}(25\lambda^2-16)\left(\frac{25\lambda^2-16}{17\lambda^2-12}\right)^{1/2}.
 \end{equation}
By virtue of \eqref{1a4}, \eqref{2a4}, \eqref{3a4}, \eqref{4a4} and \eqref{5a4}, we obtain \eqref{0aa4}. This completes the proof of Theorem \ref{tt1}.
\end{proof}

From the definitions of $\mathcal{S}_{\lambda e}^{*}$ and  $\mathcal{K}_{\lambda e}$, we know that if $f\in \mathcal{K}_{\lambda e}$, then $zf^{\prime}\in \mathcal{S}_{\lambda e}^{*}$. We thus get the following result.
\begin{theorem}\label{tt2}
If $f(z)=z+\sum_{n=2}^{\infty}a_{n}z^n\in\mathcal{K}_{\lambda e}$, then
 \begin{equation}
\abs{a_{2}}\leq\frac{1}{2}\lambda,
 \end{equation}
 \begin{equation}\label{0aa3}
\abs{a_{3}}\leq\left\{\begin{array}{lc}\frac{1}{6}\lambda &0<\lambda\leq\frac{2}{3},\\ \\
\frac{1}{4}\lambda^2 &\frac{2}{3}<\lambda\leq\frac{\pi}{2},
\end{array}\right.
 \end{equation}
 and
  \begin{equation}\label{0aa41}
 \abs{a_{4}}\leq\left\{\begin{array}{lccc}\frac{1}{12}\lambda &\lambda\in\left(0,\frac{1}{5}\right],\\ \\
\frac{1}{36}\lambda(5\lambda+2)\left(\frac{30\lambda+12}{17\lambda^2+90\lambda+12}\right)^{1/2}  &\lambda\in\left(\frac{1}{5}, r_{0}\right],\\ \\
\frac{17}{1008}(25\lambda^2-16)\left(\frac{25\lambda^2-16}{17\lambda^2-12}\right)^{1/2}  &\lambda\in\left(r_{0},\sqrt{\frac{32}{43}}\,\right],\\ \\
\frac{17}{144}\lambda^3 &\lambda\in\left(\sqrt{\frac{32}{43}},\frac{\pi}{2}\right],
\end{array}\right.
 \end{equation}
where $r_{0}\thickapprox 0.8602$ is the positive root of the equation \eqref{eq1}.
All the bounds are sharp.
\end{theorem}

From Theorem \ref{tt1}, we know that $\frac{1}{2}\abs{f^{\prime\prime}(0)}=\abs{a_2}\leq \lambda$ for $f=z+\sum_{n=2}^{\infty}a_n z^n \in \mathcal{S}_{\lambda e}^{*}$. Now, let
\begin{equation}
\mathcal{S}_{\lambda e}^{*}(\hat{p}):=\left\{f\in \mathcal{S}_{\lambda e}^{*}, f^{\prime\prime}(0)= \hat{p} \right\},
\end{equation}
where $\hat{p}$ is a given real number satisfying $-2\lambda\leq \hat{p}\leq 2\lambda$.

In what follows, we will discuss the difference of  initial successive coefficients for functions in $\mathcal{S}_{\lambda e}^{*}(\hat{p})$.

 \begin{theorem}\label{ttt2}
Let
$0\leq \hat{p}\leq 2\lambda$ and $p=\hat{p}/\lambda$. Suppose that  $f(z)=z+\sum_{n=2}^{\infty}a_{n}z^{n}\in\mathcal{S}_{\lambda e}^{*}(\hat{p})$. Then the following sharp inequalities
 \begin{equation}\label{aa32}
\abs{a_{3}-a_{2}}\leq\left\{\begin{array}{lc}\frac{1}{16}\lambda\left[8+8p-(3\lambda+2)p^2\right] &0\leq p\leq\frac{8}{3\lambda},\\ \\

\frac{1}{16}\lambda\left[8-8p+(3\lambda-2)p^2\right] &\frac{8}{3\lambda}<p\leq2.
\end{array}\right.
 \end{equation} and
  \begin{equation}\label{aaa43}
 \abs{a_{4}-a_{3}}\leq\left\{\begin{array}{lc}\Psi_{1}(\lambda,p) &0<\lambda\leq \frac{3}{5},\\ \\
\Psi_{2}(\lambda,p) &\frac{3}{5}<\lambda\leq\frac{\pi}{2},
\end{array}\right.
 \end{equation}
 hold, where
\begin{equation*}
\Psi_{1}(\lambda,p):=\left\{\begin{array}{lc}\frac{\lambda}{1152}\left[7\lambda^2p^3+(150\lambda^2+36\lambda-96)p^2+(108-360\lambda)p+600\right] &0\leq p\leq\frac{2}{4-5\lambda}, \\\\
\frac{\lambda}{288}\left[(-17\lambda^2+30\lambda-12)p^3+(54\lambda-36)p^2+(48-120\lambda)p+144\right] &\frac{2}{4-5\lambda}<p\leq2,
\end{array}\right.
 \end{equation*}
 \begin{equation*}
\Psi_{2}(\lambda,p):=\left\{\begin{array}{lc}\frac{\lambda}{1152}\left[7\lambda^2p^3+(150\lambda^2+36\lambda-96)p^2+(108-360\lambda)p+600\right] &0\leq p\leq\frac{14}{4+5\lambda}, \\\\
\frac{\lambda}{288}\left[(-17\lambda^2-30\lambda-12)p^3+(54\lambda+36)p^2+(48-120\lambda)p-144\right] &\frac{14}{4+5\lambda}<p\leq2.
\end{array}\right.
 \end{equation*}

\end{theorem}
\begin{proof}
Let $f\in \mathcal{S}_{\lambda e}^{*}(\hat{p})$. In terms of Schwarz function, we can write \eqref{a1} as
\begin{equation*}
\frac{zf^{\prime}(z)}{f(z)}=e^{\lambda\omega(z)}.
\end{equation*}
If $\varrho\in\mathcal{P}$, then it can be written in the form of Schwarz function as
\begin{equation}\label{35}
\varrho(z)=\frac{1+\omega(z)}{1-\omega(z)}=1+p_1z+p_2z^2+\cdots.
\end{equation}
From \eqref{35}, we obtain
\begin{equation}\begin{split}
\omega(z)&=\frac{\varrho(z)-1}{\varrho(z)+1}\\
&=\frac{1}{2}p_{1}z+\left(\frac{1}{2}p_{2}-\frac{1}{4}p_{1}^{2}\right)z^2+\left(\frac{1}{2}p_{3}-\frac{1}{2}p_{1}p_{2}+\frac{1}{8}p_{1}^3\right)z^3+\cdots.
\end{split}
\end{equation}
From the series expansion of $\omega$ along with some calculations, we have
\begin{equation}\label{ew}
e^{\lambda\omega(z)}=1+\frac{\lambda}{2} p_1z +\left(\frac{\lambda}{2} p_2+\frac{\lambda^2-2\lambda }{8}p_{1}^{2}\right)z^2+\left(\frac{\lambda^3-6\lambda^2+6\lambda}{48}p_{1}^{3}+\frac{\lambda^2-2\lambda}{4}p_1p_{2}+\frac{\lambda}{2}p_3\right)z^3+\cdots.
\end{equation}
Comparing \eqref{zf1} with \eqref{ew}, we get
\begin{equation}\label{b2}
a_2=\frac{1}{2}\lambda p_1,
\end{equation}
\begin{equation}\label{b3}
a_3=\frac{1}{4}\lambda\left(p_2+\frac{3\lambda-2}{4}p_{1}^2\right),
\end{equation}
and
\begin{equation}\label{b4}
a_4=\frac{1}{6}\lambda\left( p_3+\frac{5\lambda-4}{4}p_{1}p_{2}+\frac{17\lambda^2-30\lambda+12}{48}p_{1}^3\right).
\end{equation}
Let $p=\hat{p}/\lambda$. It is obvious that $p\in[0,2]$. Since $2a_2=f^{\prime\prime}(0)=\hat{p}$,
from \eqref{b2}, we get
\begin{equation}\label{p1}
p_1=\frac{1}{\lambda}\hat{p}=p.
\end{equation}
By Lemma \ref{0Lem1}, we obtain
\begin{equation}\label{p2}
p_2=\frac{1}{2}p^2+\frac{1}{2}(4-p^2)x,
\end{equation}
and
\begin{equation}\label{p3}
p_3=\frac{1}{4}p^3+\frac{1}{2}(4-p^2)px-\frac{1}{4}(4-p^2)px^2+\frac{1}{2}(4-p^2)(1-{\abs{x}}^2)y,
\end{equation}
where $x,y\in\C$ with $\abs{x}\leq1$, $\abs{y}\leq1$.
Substituting \eqref{p1}, \eqref{p2} and \eqref{p3} into \eqref{b2}, \eqref{b3} and \eqref{b4}, respectively, we obtain
\begin{equation}
a_2=\frac{1}{2}\lambda p,
\end{equation}
\begin{equation}
a_3=\frac{3}{16}\lambda^2p^2+\frac{1}{8}\lambda(4-p^2)x,
\end{equation}
and
\begin{equation}
a_4=\frac{17}{288}\lambda^3p^3+\frac{5}{48}\lambda^2(4-p^2)px-\frac{1}{24}\lambda(4-p^2)x^2+\frac{1}{12}\lambda(4-p^2)(1-{\abs{x}}^2)y
\end{equation}
for some  $x, y\in\C$ with $\abs{x}\leq1$, $\abs{y}\leq1$. For $p\in[0,\frac{8}{3\lambda}]$, we have
\begin{equation}\label{0a32}\begin{split}
\abs{a_3-a_2}&=\abs{\frac{3}{16}\lambda^2p^2+\frac{1}{8}\lambda(4-p^2)x-\frac{1}{2}\lambda p}\\
&\leq\frac{1}{2}\lambda p-\frac{3}{16}\lambda^2p^2+\frac{1}{8}\lambda(4-p^2)\\
&=\frac{1}{2}\lambda+\frac{1}{2}\lambda p-\frac{3\lambda+2}{16}\lambda p^2,
\end{split}
\end{equation}
where equality occurs if $x=-1$. Similarly, we have
\begin{equation}\label{1a32}
\abs{a_3-a_2}\leq\frac{1}{2}\lambda-\frac{1}{2}\lambda p+\frac{3\lambda-2}{16}\lambda p^2
\end{equation}
for $p\in(\frac{8}{3\lambda},2]$. Hence, the inequality \eqref{aa32} follows from \eqref{0a32} and \eqref{1a32}.
For $p=2$, we easily obtain
\begin{equation}\label{p243}
 \abs{a_4-a_3}=\frac{1}{36}\lambda^2(27-17\lambda).
 \end{equation}
 For $p\in[0,2)$, we have
\begin{equation}\label{11a43}\begin{split}
&\ \ \ \, \abs{a_{4}-a_{3}}\\&=\abs{\frac{17}{288}\lambda^3 p^3+\frac{5}{48}\lambda^2(4-p^2)px-\frac{1}{24}\lambda(4-p^2)px^2+\frac{1}{12}\lambda(4-p^2)(1-{\abs{x}}^2)y-\frac{3}{16}\lambda^2p^2-\frac{1}{8}\lambda(4-p^2)x}\\
&=\frac{1}{288}\abs{17\lambda^3p^3-54\lambda^2p^2+6\lambda(5\lambda p-6)(4-p^2)x-12\lambda(4-p^2)px^2+24\lambda(4-p^2)(1-{\abs{x}}^2)y}\\
&\leq\frac{1}{12}\lambda(4-p^2)\left[\abs{\frac{(17\lambda p-54)\lambda p^2}{24(4-p^2)}+\frac{5\lambda p-6}{4}x-\frac{p}{2}x^2}+1-{\abs{x}}^2\right]\\
&\leq\frac{1}{12}\lambda(4-p^2)Y(a,b,c),
\end{split}
\end{equation}
where $Y(a,b,c)$ is given in \eqref{yabc} and
\begin{equation}
a=\frac{(54-17\lambda p)\lambda p^2}{24(4-p^2)},\quad b=\frac{6-5\lambda p }{4},\quad c=\frac{p}{2}.
\end{equation}
Since $\lambda\in (0,\frac{\pi}{2}]$, $p\in [0,2)$, we have $54-17\lambda p>0$ and hence $a>0$.
Let $0<\lambda\leq\frac{3}{5}$. Clearly, we have $\lambda p<\frac{6}{5}$. Then it can be verified that $\abs{b}\leq 2(1-c)$ is equivalent to $0\leq p\leq\frac{2}{4-5\lambda}$. By Lemma \ref{0lem2}, we get
\begin{equation*}
Y(a,b,c)=\left\{\begin{array}{lc}\frac{7\lambda^2p^3+(150\lambda^2+36\lambda-96)p^2+(108-360\lambda)p+600}{96(4-p^2)} &0\leq p\leq\frac{2}{4-5\lambda},\\\\
\frac{(-17\lambda^2+30\lambda-12)p^3+(54\lambda-36)p^2+(48-120\lambda)p+144}{24(4-p^2)} &\frac{2}{4-5\lambda}<p<2.
\end{array}\right.
 \end{equation*}
 Thus, we obtain
\begin{equation}\label{ss1}
\abs{a_4-a_3}\leq\left\{\begin{array}{lc}\frac{\lambda}{1152}\left[7\lambda^2p^3+(150\lambda^2+36\lambda-96)p^2+(108-360\lambda)p+600\right] &0\leq p\leq\frac{2}{4-5\lambda}, \\\\
\frac{\lambda}{288}\left[(-17\lambda^2+30\lambda-12)p^3+(54\lambda-36)p^2+(48-120\lambda)p+144\right] &\frac{2}{4-5\lambda}<p<2.
\end{array}\right.
 \end{equation}
Now we suppose that $\frac{3}{5}<\lambda\leq\frac{\pi}{2}$. For $\lambda p\leq\frac{6}{5}$, we see that $\abs{b}\leq 2(1-c)$ is equivalent to
\begin{equation}\label{1pp}
(4-5\lambda) p\leq 2,
\end{equation}
it clearly holds for $\lambda\geq \frac{4}{5}$. For $\lambda\in(\frac{3}{5},\frac{4}{5})$, \eqref{1pp} is equivalent to
\begin{equation*}
p\leq \frac{2}{4-5\lambda}\in(2,+\infty),
\end{equation*}
which always holds for $p\in[0,2)$. Therefore, we have $\abs{b}\leq 2(1-c)$ provided that $ p\leq\frac{6}{5\lambda}$. If $\lambda p>\frac{6}{5}$, a simple calculation shows that $\abs{b}\leq 2(1-c)$ is equivalent to $p\leq\frac{14}{4+5\lambda}$. Since $\frac{14}{4+5\lambda}>\frac{6}{5\lambda}$ for $\lambda\in (\frac{3}{5},\frac{\pi}{2}]$. Hence, we know that $\abs{b}\leq 2(1-c)$ for $ \frac{6}{5\lambda}<p\leq\frac{14}{4+5\lambda}$. From the above discussion, we obtain $\abs{b}\leq2(1-c)$ if $p\in[0,\frac{14}{4+5\lambda}]$. Also, it is clear that $\abs{b}>2(1-c)$ if and only if
\begin{equation*}
p>{\rm max}\left\{\frac{14}{4+5\lambda}, \frac{6}{5\lambda}\right\}=\frac{14}{4+5\lambda}.
\end{equation*}
Now, an application of Lemma \ref{0lem2} leads to
\begin{equation*}
Y(a,b,c)=\left\{\begin{array}{lcc}\frac{7\lambda^2p^3+(150\lambda^2+36\lambda-96)p^2+(108-360\lambda)p+600}{96(4-p^2)} &0\leq p\leq\frac{14}{4+5\lambda},\\\\
\frac{(-17\lambda^2-30\lambda-12)p^3+(54\lambda+36)p^2+(120\lambda+48)p-144}{24(4-p^2)} &\frac{14}{4+5\lambda}<p<2.
\end{array}\right.
 \end{equation*}
From \eqref{11a43}, we deduce that
\begin{equation}\label{ss2}
\abs{a_4-a_3}\leq\left\{\begin{array}{lcc}\frac{\lambda}{1152}\left[7\lambda^2p^3+(150\lambda^2+36\lambda-96)p^2+(108-360\lambda)p+600\right] &0\leq p\leq\frac{14}{4+5\lambda}, \\\\
\frac{\lambda}{288}\left[(-17\lambda^2-30\lambda-12)p^3+(54\lambda+36)p^2+(120\lambda+48)p-144\right] &\frac{14}{4+5\lambda}<p<2.
\end{array}\right.
 \end{equation}
 Combining \eqref{p243}, \eqref{ss1} with \eqref{ss2}, we obtain \eqref{aaa43}. This completes the proof of Theorem \ref{ttt2}.
\end{proof}

Taking $\lambda=1$ in Theorem \ref{ttt2}, we obtain the following result.
\begin{cor}
Let
$0\leq p\leq 2$ and $f(z)=z+\sum_{n=2}^{\infty}a_{n}z^{n}\in\mathcal{S}_{e}^{*}(p)$. Then
 \begin{equation}\label{a32}
\abs{a_{3}-a_{2}}\leq\frac{1}{16}(-5p^2+8p+8),
 \end{equation} and
  \begin{equation}\label{a43}
 \abs{a_{4}-a_{3}}\leq\left\{\begin{array}{lcc}\frac{1}{1152}(7p^3+90p^2-252p+600) &0\leq p\leq\frac{14}{9},\\ \\
\frac{1}{288}(-59p^3+90p^2+168p-144) &\frac{14}{9}\leq p\leq2.
\end{array}\right.
 \end{equation}
All inequalities are sharp.
\end{cor}

If we denote by
\begin{equation*}
{\mathcal{S}_e}^{*}(+)=\bigcup_{0\leq p \leq 2}{\mathcal{S}_{e}^{*}(p)}=\left\{f:f\in\mathcal{S}_{e}^{*}, f^{\prime\prime}(0)=p\right\},
\end{equation*}
then in view of \eqref{a32} and \eqref{a43}, we easily obtain
\begin{equation}
\sup_{f\in{\mathcal{S}_e}^{*}(+)}\abs{a_3(f)-a_2(f)}=\frac{7}{10},
\end{equation}
and
\begin{equation}
\sup_{f\in{\mathcal{S}_e}^{*}(+)}\abs{a_4(f)-a_3(f)}=\frac{25}{48}.
\end{equation}

From Theorem \ref{tt2}, we know that for $f\in \mathcal{K}_{\lambda e}$, $\abs{\frac{1}{2}f^{\prime\prime}(0)}=\abs{a_2(f)}\leq\frac{1}{2}\lambda$. Denote by
\begin{equation}
\mathcal{K}_{\lambda e}(\hat{p})=\left\{f\in \mathcal{K}_{\lambda e}, f^{\prime\prime}(0)=\hat{p}\right\},
\end{equation}
where $\hat{p}$ is a given real number with $-\lambda\leq \hat{p}\leq \lambda$.

In what follows, we will discuss the difference of initial successive coefficients for functions belonging to the class $\mathcal{K}_{\lambda e}(\hat{p})$.
 \begin{theorem}\label{kt1}
Let
$0\leq \hat{p}\leq \lambda$ and $p=\hat{p}/\lambda$. Suppose that $f(z)=z+\sum_{n=2}^{\infty}a_{n}z^{n}$ be in the class $\mathcal{K}_{\lambda e}(\hat{p})$. Then the following sharp inequalities
 \begin{equation}\label{00a32}
\abs{a_{3}-a_{2}}\leq\frac{1}{12}\lambda\left[2+6 p-(3\lambda+2)p^2\right],
 \end{equation}
 and
  \begin{equation}\label{00a43}
 \abs{a_{4}-a_{3}}\leq\left\{\begin{array}{lc}\Theta_{1}(\lambda,p) &0<\lambda<\frac{4}{5},\\\\
\Theta_{2}(\lambda,p) &\frac{4}{5}\leq\lambda\leq\frac{\pi}{2},
\end{array}\right.
 \end{equation}
  hold, where
  \begin{equation*}
\Theta_{1}(\lambda,p):=\frac{1}{144}\lambda\left[(-17\lambda^2+30\lambda-12)p^3+(36\lambda-24)p^2+(12-30\lambda)p+24\right],
 \end{equation*}
 \begin{equation}
\Theta_{2}(\lambda,p):=\left\{\begin{array}{lc}\frac{1}{576}\lambda \left[7\lambda^2p^3+(75\lambda^2+24\lambda-48)p^2+(48-120\lambda)p+96\right] &0\leq p\leq\frac{8}{4+5\lambda},\\\\
\frac{1}{144}\lambda\left[(-17\lambda^2-30\lambda-12)p^3+(36\lambda+24)p^2+(30\lambda+12)p-24\right] &\frac{8}{4+5\lambda}<p\leq1.
\end{array}\right.
 \end{equation}
\end{theorem}
\begin{proof}
If $f\in\mathcal{K}_{\lambda e}^{*}(p)$, then
\begin{equation*}
1+\frac{zf^{\prime\prime}(z)}{f^{\prime}(z)}\prec e^{\lambda z}.
\end{equation*}
It is equivalent to
\begin{equation*}
1+\frac{zf^{\prime\prime}(z)}{f^{\prime}(z)}= e^{\lambda\omega(z)},
\end{equation*}
where $\omega$ is a Schwarz function with $\omega(0)=0$ and $\abs{\omega(z)}<1$.
From \eqref{f1}, we can write
\begin{equation}\label{cc1}
1+\frac{zf^{\prime\prime}(z)}{f^{\prime}(z)}=1+2a_2z+\left(6a_3-4a_{2}^{2}\right)z^2+\left(12a_4-18a_2a_3+8a_{2}^{3}\right)z^3+\cdots.
\end{equation}
From \eqref{cc1} and \eqref{ew}, we have
\begin{equation}\label{aa2}
a_2=\frac{\lambda}{4}p_1,
\end{equation}
\begin{equation}\label{aa3}
a_3=\frac{1}{12}\lambda\left(p_2+\frac{3\lambda-2}{4}{p_1}^2\right),
\end{equation}
and
\begin{equation}\label{aa4}
a_4=\frac{1}{24}\lambda\left(p_3 +\frac{5\lambda-4}{4}p_{1}p_{2}+\frac{17\lambda^2-30\lambda+12}{48}p_{1}^3\right).
\end{equation}
Let $p=\hat{p}/\lambda$. It is clear that $p\in [0,1]$. Since $2a_2=f^{\prime\prime}(0)=\hat{p}$,
in view of \eqref{aa2}, we get
\begin{equation}\label{pp1}
p_1=\frac{2}{\lambda}\hat{p}=2p.
\end{equation}
By Lemma \ref{0Lem1}, we obtain
\begin{equation}\label{pp2}
p_2=2p^2+2(1-p^2)x,
\end{equation}
and
\begin{equation}\label{pp3}
p_3=2p^3+4(1-p^2)px-2(1-p^2)px^2+2(1-p^2)(1-{\abs{x}}^2)y,
\end{equation}
where $x, y\in\C$ with $\abs{x}\leq1$, $\abs{y}\leq1$.
Substituting \eqref{pp1}, \eqref{pp2} and \eqref{pp3} into \eqref{aa2}, \eqref{aa3} and \eqref{aa4}, respectively, we have
\begin{equation}
a_2=\frac{\lambda}{2}p,
\end{equation}
\begin{equation}
a_3=\frac{1}{4}\lambda^2p^2+\frac{1}{6}\lambda(1-p^2)x,
\end{equation}
and
\begin{equation}
a_4=\frac{17}{144}\lambda^3p^3+\frac{5}{24}\lambda^2(1-p^2)px-\frac{1}{12}\lambda(1-p^2)px^2+\frac{1}{12}\lambda(1-p^2)(1-{\abs{x}}^2)y
\end{equation}
for some  $x,y\in\C$ and $\abs{x}\leq1$, $\abs{y}\leq1$. Thus, we find that
\begin{equation}\label{ka32}\begin{split}
\abs{a_3-a_2}&=\abs{\frac{1}{4}\lambda^2p^2+\frac{1}{6}\lambda(1-p^2)x-\frac{1}{2}\lambda p}\\
&=\abs{\frac{1}{4}\lambda p(\lambda p-2)+\frac{1}{6}\lambda(1-p^2)x}\\
&\leq\frac{1}{6}\lambda+\frac{1}{2}\lambda p-\frac{3\lambda+2}{12}\lambda p^2,
\end{split}
\end{equation}
where equality occurs if $x=-1$. The inequality \eqref{00a32} in Theorem \ref{kt1} follows from \eqref{ka32}. For $p=1$, we easily obtain that
\begin{equation}\label{0a43}
\abs{a_4-a_3}=\frac{1}{144}\lambda^2(36-17\lambda).
\end{equation}
For $p\in[0,1)$, we have
\begin{equation*}\begin{split}
&\ \ \ \,\abs{a_{4}-a_{3}}\\&=\abs{\frac{17}{144}\lambda^3p^3+\frac{5}{24}\lambda^2(1-p^2)px-\frac{1}{12}\lambda(1-p^2)px^2+\frac{1}{12}\lambda(1-p^2)(1-{\abs{x}}^2)y-\frac{1}{4}\lambda^2p^2-\frac{1}{6}\lambda(1-p^2)x}\\
&=\frac{1}{144}\lambda\abs{(17\lambda p-36)\lambda p^2+6(5\lambda p-4)(1-p^2)x-12(1-p^2)px^2+12(1-p^2)(1-{\abs{x}}^2)y}\\
&\leq\frac{1}{12}\lambda(1-p^2)\left[\abs{\frac{(17\lambda p-36)\lambda p^2}{12(1-p^2)}+\frac{5\lambda p-4}{2}x-px^2}+1-{\abs{x}}^2\right]\\
&\leq\frac{1}{12}\lambda(1-p^2)Y(a,b,c),
\end{split}
\end{equation*}
where $Y(a,b,c)$ is given in \eqref{yabc} and
\begin{equation*}
a=\frac{(36-17\lambda p)\lambda p^2}{12(1-p^2)},\quad b=\frac{4-5\lambda p }{2},\quad c=p.
\end{equation*}
For $\lambda\in (0,\frac{\pi}{2}]$, $p\in [0,1)$, it can be seen that $a>0$.  Since  $\abs{b}\leq 2(1-c)$ is equivalent to $b^2\leq4(1-c)^2$, we see that
 $\abs{b}\leq 2(1-c)$ if and only if
 \begin{equation}\label{keq}
(4-5\lambda)\left(1-\frac{5\lambda+4}{8}p\right)\leq0.
 \end{equation}
Clearly, \eqref{keq} holds only if $\lambda\geq\frac{4}{5}$. This means  that $\abs{b}>2(1-c)$ for all $\lambda\in (0,\frac{4}{5})$. Let $\lambda\in (0,\frac{4}{5})$.  By using Lemma \ref{0lem2} we have
\begin{equation*}
Y(a,b,c)=\frac{(-17\lambda^2+30\lambda-12)p^3+(36\lambda-24)p^2+(12-30\lambda)p+24}{12(1-p^2)}.
 \end{equation*}
 This induces that
 \begin{equation}\label{0b43}
\abs{a_4-a_3}\leq\frac{1}{144}\lambda\left[(-17\lambda^2+30\lambda-12)p^3+(36\lambda-24)p^2+(12-30\lambda)p+24\right].
 \end{equation}
Now, we suppose that $\frac{4}{5}\leq\lambda\leq\frac{\pi}{2}$. From \eqref{keq}, we obtain $\abs{b}\leq 2(1-c)$ if and only if $p\in [0,\frac{8}{4+5\lambda}]$.  By noting that $\frac{8}{4+5\lambda}\geq\frac{4}{5\lambda}$ for $\lambda\in [\frac{4}{5},\frac{\pi}{2}]$. An application of Lemma \ref{0lem2} yields
\begin{equation*}
Y(a,b,c)=\left\{\begin{array}{lc}\frac{7\lambda^2p^3+(75\lambda^2+24\lambda-48)p^2+(48-120\lambda)p+96}{48(1-p^2)} &0\leq p\leq\frac{8}{4+5\lambda},\\\\
\frac{(-17\lambda^2-30\lambda-12)p^3+(36\lambda+24)p^2+(30\lambda+12)p-24}{12(1-p^2)} &\frac{8}{4+5\lambda}<p<1.
\end{array}\right.
 \end{equation*}
Thus, we deduce that
\begin{equation}\label{1b43}
\abs{a_4-a_3}\leq\left\{\begin{array}{lc}\frac{1}{576}\lambda \left[7\lambda^2p^3+(75\lambda^2+24\lambda-48)p^2+(48-120\lambda)p+96\right] &0\leq p\leq\frac{8}{4+5\lambda},\\\\
\frac{1}{144}\lambda\left[(-17\lambda^2-30\lambda-12)p^3+(36\lambda+24)p^2+(30\lambda+12)p-24\right] &\frac{8}{4+5\lambda}<p<1.
\end{array}\right.
 \end{equation}
 Combining \eqref{0a43}, \eqref{0b43} with \eqref{1b43}, we obtain the inequality \eqref{00a43} in Theorem \ref{kt1}. The proof is thus completed.
\end{proof}

By choosing $\lambda=1$ in Theorem \ref{kt1}, we obtain the following result.
\begin{cor}\label{c2}
Let
$0\leq p\leq 1$. Suppose that $f(z)=z+\sum_{n=2}^{\infty}a_{n}z^{n}$ be in the class $\mathcal{K}_{e}(p)$. Then the following sharp inequalities
 \begin{equation}\label{bb1}
\abs{a_{3}-a_{2}}\leq\frac{1}{12}(-5p^2+6p+2),
 \end{equation}
 and
  \begin{equation}\label{bb2}
 \abs{a_{4}-a_{3}}\leq\left\{\begin{array}{lc}\frac{1}{576}\left(7p^3+51p^2-72p+96\right) &0\leq p\leq\frac{8}{9},\\\\
\frac{1}{144}\left(-59p^3+60p^2+42p-24\right)&\frac{8}{9}<p\leq1,
\end{array}\right.
 \end{equation}
  hold.
\end{cor}
Now, we denote by
\begin{equation}
{\mathcal{K}_e}^{*}(+)=\bigcup_{0\leq p \leq 1}{\mathcal{K}_{e}(p)}=\left\{f:f\in\mathcal{K}_{e}, f^{\prime\prime}(0)=p\right\}.
\end{equation}
By virtue of \eqref{bb1} and \eqref{bb2}, we easily find that
\begin{equation}
\sup_{f\in{\mathcal{K}_e}^{*}(+)}\abs{a_3(f)-a_2(f)}=\frac{19}{60},
\end{equation}
and
\begin{equation}
\sup_{f\in{\mathcal{K}_e}^{*}(+)}\abs{a_4(f)-a_3(f)}=\frac{1}{6}.
\end{equation}

Finally, we will give the upper bounds of $\abs{a_{n}}(n\geq2)$ for functions in the class $\mathcal{S}_{\lambda e}^{*}$ and $\mathcal{K}_{\lambda e}$. However, they are not always sharp.
 \begin{theorem}\label{t1}
If $f(z)=z+\sum_{n=2}^{\infty}a_{n}z^n\in\mathcal{S}_{\lambda e}^{*}$, then
 \begin{equation}\label{a211}
\abs{a_{n}}\leq\frac{1}{(n-1)!}\prod_{k=0}^{n-2}(\lambda+k)\quad (n\geq2).
 \end{equation}
\end{theorem}
\begin{proof}
Let
\begin{equation*}
\psi(z):=\frac{zf^{\prime}(z)}{f(z)}=1+\sum_{n=1}^{\infty}c_{n}z^{n}.
\end{equation*}
Since $f\in \mathcal{S}_{\lambda e}^{*}$, we know that
\begin{equation*}
\psi(z)\prec \chi(z)=e^{\lambda z}=1+\lambda z+\frac{\lambda^2}{2!}z^2+\cdots.
\end{equation*}
Since $\chi(z)$ is univalent and convex in $\D$ for $0<\lambda\leq\frac{\pi}{2}$, by Lemma \ref{Lem02}, we obtain
\begin{equation*}
\abs{c_{n}}\leq \lambda\quad (n\geq1).
\end{equation*}
In view of $zf^{\prime}(z)=\psi(z)f(z)$, by comparing the coefficients of $z^{n}$ on both sides, it follows that
\begin{equation*}
a_{2}=c_{1}
\end{equation*}
and
\begin{equation*}
(n-1)a_{n}=c_{n-1}+\sum_{k=2}^{n-1}c_{n-k}a_{k}\quad(n\geq3).
\end{equation*}
Thus, we have
\begin{equation*}
\abs{a_{2}}=\abs{c_{1}}\leq\lambda,
\end{equation*}
and
\begin{equation}\label{219}
\abs{a_{n}}\leq\frac{\lambda}{n-1}\left(1+\sum_{k=2}^{n-1}\abs{a_{k}}\right)\quad(n\geq3).
\end{equation}
Now, we define the
sequence $\{A_{m}\}_{m=2}^{\infty}$ as follows:
 \begin{equation}\label{220}
 \begin{cases}
\displaystyle A_{m}=\lambda&(m=2),\\
\displaystyle A_{m}=\frac{\lambda}{m-1}\left(1+\sum_{k=2}^{m-1}A_{k}\right)&\left(m\geq3\right).
 \end{cases}
 \end{equation}
 In order to prove that
 \begin{equation}
 \abs{a_{m}}\leq A_{m}\quad(m\geq2),
 \end{equation}
 we use the principle of mathematical induction. It is easy to verify that
 \begin{equation}
 \abs{a_{2}}\leq A_{2}=\lambda.
 \end{equation}
 Thus, assuming that
 \begin{equation}\label{223}
 \abs{a_{l}}\leq A_{l}\quad(l=2,3,\ldots,m),
 \end{equation}
 we find from \eqref{219} and \eqref{223} that
 \begin{equation}
 \abs{a_{m+1}}\leq\frac{\lambda}{m}\left(1+\sum_{k=2}^{m}\abs{a_{k}}\right)\leq\frac{\lambda}{m}\left(1+\sum_{k=2}^{m}A_{k}\right)=A_{m+1}.
 \end{equation}
 Therefore, by the principle of mathematical induction, we have
 \begin{equation}\label{227}
 \abs{a_{m}}\leq A_{m}\quad(m\geq2).
 \end{equation}
 By means of Lemma \ref{Lem5} and \eqref{220}, we see that
 \begin{equation}\label{228}
 A_{m}=\frac{1}{(m-1)!}\prod_{k=0}^{m-2}(\lambda+k)\quad\left(m\geq2\right).
 \end{equation}
 Combining \eqref{227} with \eqref{228}, we readily get the coefficient estimates \eqref{a211} asserted by Theorem \ref{t1}.
 \end{proof}
According to the relationship between the classes $\mathcal{S}_{\lambda e}^{*}$ and $\mathcal{K}_{\lambda e}$, we easily obtain the follow result.
 \begin{theorem}\label{kkt1}
If $f(z)=z+\sum_{n=2}^{\infty}a_{n}z^n\in\mathcal{K}_{\lambda e}$, then
 \begin{equation*}
\abs{a_{n}}\leq\frac{1}{n!}\prod_{k=0}^{n-2}(\lambda+k)\quad (n\geq2).
 \end{equation*}
\end{theorem}
\vskip.12in

\begin{center}{\sc Acknowledgments}
\end{center}
\vskip.03in

L. Shi was supported by the \textit{Foundation for Excellent Youth Teachers of Colleges and Universities of Henan  Province} under Grant no. 2019GGJS195 and the \textit{Key Project of Natural Science Foundation of
Educational Committee of Henan Province} under Grant no. 20B110001 of the P. R. China. Z.-G. Wang was supported by the \textit{Key Project of Education Department of Hunan Province} under Grant no.
19A097 of the P. R. China. The authors would like to thank the referees for their valuable comments and suggestions, which was essential to improve the quality of this paper.

\end{document}